\documentclass{article}

\usepackage{amsmath, amssymb, amsthm}
\usepackage{mathtools}
\usepackage{geometry}
\usepackage{paralist}
\usepackage{indentfirst}
\usepackage{tikz}

\DeclarePairedDelimiter\abs{\lvert}{\rvert}
\DeclarePairedDelimiter\card{\lvert}{\rvert}
\DeclarePairedDelimiter\order{\lvert}{\rvert}

\newcommand{\indicator}[1]{{\bf 1}_{#1}}

\newcommand{\PP}{\mathcal{P}}
\newcommand{\QQ}{\mathcal{Q}}
\newcommand{\spider}{S}

\theoremstyle{plain}
\newtheorem{theorem}{Theorem}[section]
\newtheorem{lemma}[theorem]{Lemma}
\newtheorem{corollary}[theorem]{Corollary}
\newtheorem{conjecture}[theorem]{Conjecture}

\theoremstyle{definition}
\newtheorem*{remark}{Remark}

\usepackage[pdftex, colorlinks]{hyperref}

\title{The tree search game for two players}

\author{
Ravi B. Boppana\thanks{Department of Mathematics, Massachusetts Institute of Technology, Cambridge, Massachusetts, USA\@.
Email address: {\tt rboppana@mit.edu} }
\and 
Joel Brewster Lewis\thanks{Department of Mathematics, The George Washington University,
Washington, DC, USA\@.  
Email address: {\tt jblewis@gwu.edu} }
}

\begin{document}

\maketitle

\begin{abstract}
We consider a two-player search game on a tree $T$.  One vertex (unknown to the players) is randomly selected as the target.  The players alternately guess vertices.  If a guess $v$ is not the target, then both players are informed in which subtree of $T \smallsetminus v$ the target lies.  The winner is the player who guesses the target. 

When both players play optimally, we show that each of them wins with probability approximately $1/2$.  When one player plays optimally and the other plays randomly, we show that the player with the optimal strategy wins with probability between $9/16$ and $2/3$ (asymptotically).  When both players play randomly, we show that each wins with probability between $13/30$ and $17/30$ (asymptotically).
\end{abstract}

\section{Introduction}

We consider the following competitive variant of traditional binary search: 
two players seek an (unknown, uniformly random) element of 
the set~$\{1, \dots, n\}$.  The players alternately guess elements of the set; if a guess is incorrect, then both players are informed whether the secret number is larger or smaller than the guess.  The winner is the player who guesses the secret number.

We consider the following more general variant of the game, using a model of binary search on trees introduced by Onak and Parys~\cite{OP}: 
the starting position is not a set of $n$~numbers, but rather a labeled tree~$T$ on $n$~vertices, one vertex of which has been selected (uniformly at random) as the target.  
The two players alternately choose vertices; if a guess~$v$ is incorrect, 
then both players are informed in which subtree of~$T \smallsetminus v$ the target vertex lies.
The winner is the player who guesses the target vertex.  
One immediately recovers the previous game upon choosing $T$ to be a path on $n$~vertices.

We consider the variant of the game in which both players play according to an optimal strategy, as well as variants where one or both players play uniformly at random.  When both players play optimally, we completely analyze the game: 
if $n$ (the number of vertices) is even, then the game is fair, 
while if $n$ is odd, then the first player wins with probability~$\frac{1}{2} + \frac{1}{2n}$, regardless of the structure of the tree.  We also describe all optimal strategies in this case.  

When one or both players play randomly
(selecting a vertex uniformly at random from among the vertices that could be the target),  
the probabilities of winning are surprisingly more complicated to analyze.  
We complete the analysis in the case of paths and stars, 
and conjecture that the probability of a first-player win always lies between these two extremes (Conjectures~\ref{conj:half random tree bounds} and~\ref{conj:all random}); 
we are able to establish the conjectures in some cases.
Specifically, when one player plays optimally and the other plays randomly,
we show that the optimally playing competitor wins with probability between $9/16$ and $2/3$ (asymptotically). 
When both players play uniformly at random, 
we show that each wins with probability between $13/30$ and $17/30$ (for~$n \ge 2$).

The structure of the paper is as follows: in Section~\ref{sec:background}, we establish the terminology and notation that is used throughout the paper as well as some basic lemmas.  In Section~\ref{sec:strategic}, we completely analyze the game in the case of two optimally playing competitors.  In Section~\ref{sec:mixed}, we study the game in the case that one player plays according to an optimal strategy while the other chooses vertices uniformly at random.  In Section~\ref{sec:random}, we study the game in the case that \emph{both} players play uniformly at random.  Finally, in Section~\ref{sec:concluding remarks}, we give a number of open problems including variants of the game that we believe might be of interest.

\section{Background and notation}
\label{sec:background}

We begin by establishing some basic terminology and notation for the rest of the paper.

As usual, a \emph{tree} is a connected acyclic graph.
Our trees are undirected and unrooted, with a finite but positive number of vertices.
We denote by $V(T)$ the vertex set of tree~$T$.
The \emph{order}~$\order{T}$ of tree~$T$ is~$\card{V(T)}$,
the number of vertices of~$T$.  
We denote by $E(T)$ the edge set of tree~$T$.
We denote by $v \sim w$ the relation that vertices $v$ and $w$ are joined by an edge of the tree.

The \emph{degree}~$\deg(v)$ of a vertex~$v$ of a tree~$T$
is the number of edges of~$T$ incident to~$v$.
A \emph{leaf} is a vertex of degree~$1$.
Every tree $T$ of order~$n$ has $n - 1$~edges, 
and the sum of the vertex degrees of~$T$ is~$2(n - 1)$.
Given a vertex~$v$ of~$T$, 
let $T \smallsetminus v$ be the graph that results from deleting~$v$ and its edges from~$T$;
the graph~$T \smallsetminus v$ is a forest with $\deg(v)$~components, 
each of which is a tree.  
Given an edge~$e$ of~$T$, 
let $T \smallsetminus e$ be the graph that results from deleting~$e$ from~$T$;
the graph~$T \smallsetminus e$ is a forest with two components, 
each of which is a tree.

Denote by $P_n$ the path graph on $n$~vertices and by $S_n$ the star graph on $n$~vertices. 
In particular, $P_1 = S_1$, $P_2 = S_2$, and $P_3 = S_3$ are the unique trees on $1$, $2$, and $3$ vertices, respectively.

At several points, we will be concerned with counting leaves and other small subtrees near ``the boundary'' of a given tree.  To that end, define for each tree~$T$ and each integer~$k$ the \emph{limb set}
\[
L_k(T) 
  = \left\{ (v, T') \colon\, v \in V(T) \textrm{ and } 
	             T' \textrm{ is a component of } T \smallsetminus v \textrm{ of order } k \right\}
\]
and the \emph{limb number}
\[
\ell_k(T) = \card{L_k(T)}.
\] 
If $k \le 0$ or $k \ge \order{T}$, then $\ell_k(T) = 0$.
We record below some basic information about the limb numbers.

\begin{lemma} \label{lemma:ell}
Let $T$ be a tree on $n$~vertices.
\begin{enumerate}[(a)]

\item
$\ell_1(T)$ is the number of leaves of~$T$.

\item 
If $n > 1$, then $\ell_1(T) \geq 2$.

\item 
For every integer~$k$, we have $\ell_k(T) = \ell_{n - k}(T)$.

\item
The sum $\sum_k \ell_k(T)$ is $2(n - 1)$.

\item
The sum $\sum_k k \ell_k(T)$ is $n(n - 1)$.

\item 
For every integer~$k$, we have $\ell_1(T) \geq \ell_k(T)$.

\item
If $T$ is not the tree in Figure~\ref{fig:counter-example}, 
then $\ell_1(T) + \ell_3(T) \ge \ell_2(T) + \ell_4(T)$.
\begin{figure}
\begin{center}
\begin{tikzpicture}[node distance = 2cm, v/.style = {circle, draw}]
\node[v] (1)                {};
\node[v] (2) [right of = 1] {};
\node[v] (3) [right of = 2] {};
\node[v] (4) [right of = 3] {};
\node[v] (5) [right of = 4] {};
\draw (1) -- (2) -- (3) -- (4) -- (5);
\node[v] (6) [below of = 3] {};
\draw (3) -- (6);
\end{tikzpicture}
\end{center}
\caption{The unique tree for which the inequality in Lemma~\ref{lemma:ell}(g) fails}
\label{fig:counter-example}
\end{figure}
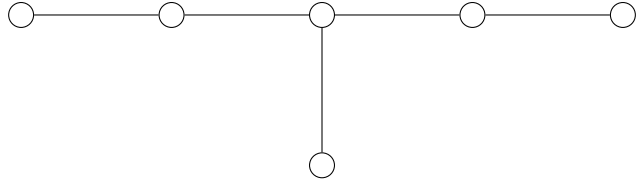

\end{enumerate}
\end{lemma}

\begin{proof} \
\begin{enumerate}[(a)]

\item
The map that sends the leaf~$w$ to the pair~$(v, T')$ where~$v$ is the unique neighbor of~$w$ and $T'$ is the tree with vertex set~$\{w\}$ is a bijection between the set of leaves of~$T$ and~$L_1(T)$.
Hence $\ell_1(T)$ is the number of leaves of~$T$.

\item
It is a standard result that may be found in many textbooks on graph theory that 
every finite tree with more than one vertex has at least two leaves; 
for example, see~\cite[Lemma 2.1.3]{West}.
A short proof: 
since the $n$ vertex degrees are positive integers and their sum is~$2(n - 1)$,
at least two of them must be equal to~$1$.  

\item
We index the (disjoint) union~$\bigcup_k L_k(T)$ as follows: 
if $v$ is a vertex of~$T$ and $w$ is a neighbor of~$v$, 
then let $T_{v, w}$ be the component of~$T \smallsetminus v$ that contains~$w$.  
Thus
\[
  L_k(T) = \{(v, T_{v, w}) \colon\, v \sim w \textrm{ and } \order{T_{v, w}} = k\}.
\]
Now fix an edge~$e = \{v, w\}$ of~$T$.    
Since $T_{v, w} \cup T_{w, v} = T \smallsetminus e$, 
one has  $\order{T_{v, w}} + \order{T_{w, v}} = n$.
It follows that $(v, T_{v, w}) \in L_k(T)$ if and only if $(w, T_{w, v}) \in L_{n - k}(T)$, and hence $\ell_k(T) = \ell_{n - k}(T)$.

\item
The map that sends the pair~$(v, T_{v, w})$ to the edge~$\{v, w\}$
is a two-to-one function from the union~$\bigcup_k L_k(T)$ to the edge set~$E(T)$.  
Hence $\sum_k \ell_k(T) = 2 \card{E(T)} = 2(n - 1)$.

\item
By part (c), we have
\[
  2 \sum_k k \ell_k(T)
	  = \sum_k k \ell_k(T) + \sum_k k \ell_{n - k}(T)
		= \sum_k k \ell_k(T) + \sum_k (n - k) \ell_k(T)
		= n \sum_k \ell_k(T) .
\]
Dividing by $2$ and using part (d) gives
\[
  \sum_k k \ell_k(T)
	  = \frac{n}{2} \sum_k \ell_k(T)
		= \frac{n}{2} \cdot 2(n - 1)
		= n(n - 1).
\]

\item
If~$k \le 0$, then $\ell_k(T) = 0 \le \ell_1(T)$.
If~$k = 1$, then $\ell_k(T) = \ell_1(T)$.  
Hence we may assume that $k > 1$.
By part~(c), we may further assume that~$k \le \frac{n}{2}$.

Suppose that $(v', T')$ and $(v'', T'')$ are distinct elements of~$L_k(T)$.
We claim that $T'$ and $T''$ are vertex-disjoint.
This clearly holds when $v' = v''$, so we may assume that $v' \ne v''$.
We have $T' = T_{v', w'}$ for some neighbor~$w'$ of~$v'$
and $T'' = T_{v'', w''}$ for some neighbor~$w''$ of~$v''$.
Are $w'$ and/or~$w''$ on the (unique) path between $v'$ and~$v''$ in~$T$? 
The four cases are illustrated in Figure~\ref{fig:four cases}.
\begin{figure}
\begin{center}
\begin{tikzpicture}[node distance = 2cm, v/.style = {circle, draw}]
\node    (0)                {Case 1:};
\node[v] (1) [right of = 0] {$v'$};
\node[v] (2) [right of = 1] {\color{blue} $w'$};
\node    (3) [right of = 2] {$\, \cdots \,$};
\node[v] (4) [right of = 3] {$v''$};
\node[v] (5) [right of = 4] {\color{blue} $w''$};
\draw (1) -- (2) -- (3) -- (4) -- (5);
\end{tikzpicture}
\begin{tikzpicture}[node distance = 2cm, v/.style = {circle, draw}]
\node    (0)                {Case 2:};
\node[v] (1) [right of = 0] {\color{blue} $w'$};
\node[v] (2) [right of = 1] {$v'$};
\node    (3) [right of = 2] {$\, \cdots \,$};
\node[v] (4) [right of = 3] {\color{blue} $w''$};
\node[v] (5) [right of = 4] {$v''$};
\draw (1) -- (2) -- (3) -- (4) -- (5);
\end{tikzpicture}
\begin{tikzpicture}[node distance = 2cm, v/.style = {circle, draw}]
\node    (0)                {Case 3:};
\node[v] (1) [right of = 0] {$v'$};
\node[v] (2) [right of = 1] {\color{blue} $w'$};
\node    (3) [right of = 2] {$\, \cdots \,$};
\node[v] (4) [right of = 3] {\color{blue} $w''$};
\node[v] (5) [right of = 4] {$v''$};
\draw (1) -- (2) -- (3) -- (4) -- (5);
\end{tikzpicture}
\begin{tikzpicture}[node distance = 2cm, v/.style = {circle, draw}]
\node    (0)                {Case 4:};
\node[v] (1) [right of = 0] {\color{blue} $w'$};
\node[v] (2) [right of = 1] {$v'$};
\node    (3) [right of = 2] {$\, \cdots \,$};
\node[v] (4) [right of = 3] {$v''$};
\node[v] (5) [right of = 4] {\color{blue} $w''$};
\draw (1) -- (2) -- (3) -- (4) -- (5);
\end{tikzpicture}
\end{center}
\caption{The four cases in the proof of Lemma~\ref{lemma:ell}(f)}
\label{fig:four cases}
\end{figure}
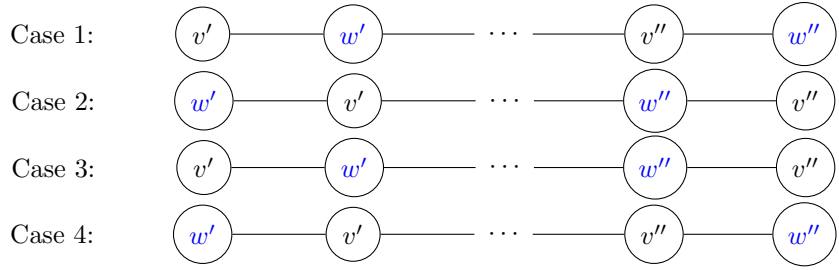
If either $w'$ or~$w''$ (but not both) were on the path between $v'$ and~$v''$ (Case~1 or~2), 
then one of the trees $T'$ and~$T''$ would strictly contain the other, 
which is impossible since they have equal orders. 
If both $w'$ and~$w''$ are on the path (Case~3), 
which includes as special cases the situation~$w' = w''$ and the situation $w' = v''$ and~$w'' = v'$,
then $V(T') \cup V(T'') = V(T)$,
while $\order{T'} + \order{T''} = 2k \le n$;
hence $V(T')$ and~$V(T'')$ partition~$V(T)$ 
and in particular are disjoint.
If neither $w'$ nor $w''$ is on the path (Case~4),
then $T'$ and~$T''$ are disjoint.
In every case, $T'$ and $T''$ are disjoint.  

By part~(b), since~$k > 1$, 
for every pair~$(v, T_{v, w})$ in~$L_k(T)$,
there is a vertex~$\phi(v, T_{v, w})$ different from~$w$ 
that is a leaf of subtree~$T_{v, w}$ and hence a leaf of~$T$.
By vertex disjointness, 
the map~$\phi$ is an injective function from~$L_k(T)$ to the set of leaves of~$T$.
Consequently, by part~(a), we have $\ell_k(T) \leq \ell_1(T)$. 
 
\item 
If $n \le 4$, then $\ell_4(T) = 0$ and part~(f) implies that $\ell_1(T) \ge \ell_2(T)$.  
If $n = 5$, then part~(c) implies that $\ell_1(T) = \ell_4(T)$ and $\ell_2(T) = \ell_3(T)$.
If $n = 6$, since $T$ is not the tree in Figure~\ref{fig:counter-example}, $T$ is one of the trees in Figure~\ref{fig:all six vertices}, and one checks by hand that $\ell_1(T) + \ell_3(T) \geq \ell_2(T) + \ell_4(T)$ in all cases.
\begin{figure}
\begin{center}
\raisebox{.3cm}{\scalebox{.5}{\begin{tikzpicture}[node distance = 2cm, v/.style = {circle, draw}]
\node[v] (1)                {};
\node[v] (2) [right of = 1] {};
\node[v] (3) [right of = 2] {};
\node[v] (4) [right of = 3] {};
\draw (1) -- (2) -- (3) -- (4);
\node[v] (5) [above right of = 4] {};
\draw (4) -- (5);
\node[v] (6) [below right of = 4] {};
\draw (4) -- (6);
\end{tikzpicture}}}
\qquad\qquad
\scalebox{.5}{\begin{tikzpicture}[node distance = 2cm, v/.style = {circle, draw}]
\node[v] (1)                {};
\node[v] (2) [right of = 1] {};
\node[v] (3) [right of = 2] {};
\node[v] (4) [right of = 3] {};
\draw (1) -- (2) -- (3) -- (4);
\node[v] (5) [above of = 3] {};
\node[v] (6) [below of = 3] {};
\draw (5) -- (3) -- (6);
\end{tikzpicture}}
\qquad\qquad
\scalebox{.5}{\begin{tikzpicture}[node distance = 2cm, v/.style = {circle, draw}]
\draw (-2, 0) node[v, fill = white]{} -- (0, 0) -- (-.62, 1.9) node[v, fill = white]{};
\draw (-.62, -1.9) node[v, fill = white]{} -- (0, 0) -- (1.62, 1.18) node[v, fill = white]{};
\draw (1.62, -1.18)  node[v, fill = white]{} -- (0, 0)  node[v, fill = white]{};
\end{tikzpicture}}

\bigskip

\scalebox{.5}{\begin{tikzpicture}[node distance = 2cm, v/.style={circle,draw}]
\node[v] (1) {};
\node[v] (2) [below right of = 1] {};
\node[v] (3) [below left of = 2] {};
\node[v] (7) [right of = 2] {};
\node[v] (8) [above right of = 7] {};
\node[v] (9) [below right of = 7] {};
\draw (1) -- (2);
\draw (3) -- (2) -- (7) -- (8);
\draw (7) -- (9);
\end{tikzpicture}}
\qquad\qquad
\raisebox{.7cm}{\scalebox{.5}{\begin{tikzpicture}[node distance = 2cm, v/.style = {circle, draw}]
\node[v] (1)                {};
\node[v] (2) [right of = 1] {};
\node[v] (3) [right of = 2] {};
\node[v] (4) [right of = 3] {};
\node[v] (5) [right of = 4] {};
\node[v] (6) [right of = 5] {};
\draw (1) -- (2) -- (3) -- (4) -- (5) -- (6);
\end{tikzpicture}}}
\end{center}
\caption{The trees with six vertices other than the tree in Figure~\ref{fig:counter-example}, with associated limb numbers $(\ell_1(T), \ell_2(T), \ell_3(T), \ell_4(T)) = (3, 1, 2, 1)$, $(4, 1, 0, 1)$, $(5, 0, 0, 0)$, $(4, 0, 2, 0)$, and $(2, 2, 2, 2)$}
\label{fig:all six vertices}
\end{figure}
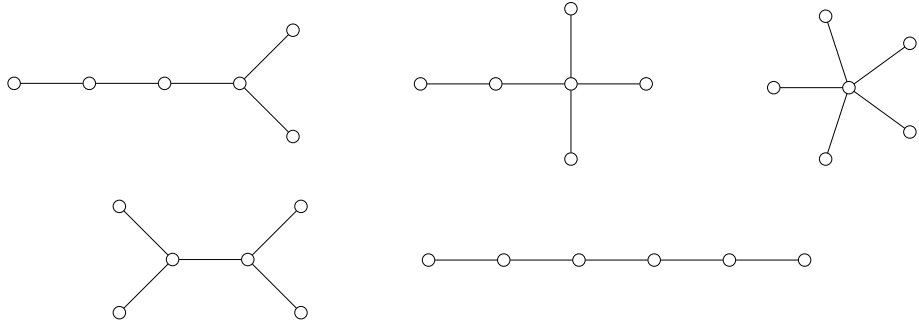 
%If $n \le 4$, then $\ell_4(T) = 0$ and part~(f) implies that $\ell_1(T) \ge \ell_2(T)$.  
%If $n = 5$, then part~(c) implies that $\ell_1(T) = \ell_4(T)$ and $\ell_2(T) = \ell_3(T)$.
%If $n = 6$ and $\ell_1(T) \le 2$, then $T$ is the path~$P_6$,
%for which $\ell_1(T)$, $\ell_2(T)$, $\ell_3(T)$, and $\ell_4(T)$ all equal~$2$.  
%If $n = 6$ and $\ell_1(T) = 3$, then (since $T$ is not the tree in Figure~\ref{fig:counter-example}) $T$ is the tree in Figure~\ref{fig:six vertices},
%for which $\ell_2(T) = 1$ and $\ell_4(T) = 1$. 
%\begin{figure}
%\begin{center}
%\begin{tikzpicture}[node distance = 2cm, v/.style = {circle, draw}]
%\node[v] (1)                {};
%\node[v] (2) [right of = 1] {};
%\node[v] (3) [right of = 2] {};
%\node[v] (4) [right of = 3] {};
%\draw (1) -- (2) -- (3) -- (4);
%\node[v] (5) [above right of = 4] {};
%\draw (4) -- (5);
%\node[v] (6) [below right of = 4] {};
%\draw (4) -- (6);
%\end{tikzpicture}
%\end{center}
%\caption{A tree with $6$ vertices and $3$ leaves used in the proof of Lemma~\ref{lemma:ell}(g)}
%\label{fig:six vertices}
%\end{figure} 
%If $n = 6$ and $\ell_1(T) \ge 4$, then parts~(c) and~(d) imply that
%\[
%  \ell_2(T) + \ell_4(T)
%	  = 10 - \ell_1(T) - \ell_3(T) - \ell_5(T)
%		\le 10 - 2 \ell_1(T) 
%		< \ell_1(T) .
%\] 
If $n = 7$, then parts~(c) and~(f) imply that 
$\ell_3(T) = \ell_4(T)$ and $\ell_1(T) \geq \ell_2(T)$.  
Now suppose that~$n \ge 8$.  
We construct an injective map $\phi: L_2(T) \cup L_4(T) \hookrightarrow L_1(T) \cup L_3(T)$.  Given a pair $(v, T') \in L_2(T) \cup L_4(T)$, let $w$ be the neighbor of~$v$ in~$T'$.  
The graph~$T' \smallsetminus w$ is a forest on an odd number of vertices, 
so it has a connected component~$T''$ of odd order; 
define $\phi(v, T') = (w, T'')$.  
By construction, this map has the correct domain and range; 
it remains to show that it is injective.  
Suppose otherwise, 
so that $\phi(v_1, T_1) = \phi(v_2, T_2) = (w, T'')$ and $(v_1, T_1) \ne (v_2, T_2)$. 
Then by the definition of~$\phi$, 
the tree~$T$ has the form illustrated in Figure~\ref{fig:tree form}, namely, $v_1$ is adjacent to $w$ by an edge $e_1$ and $T_1$ is the component of containing $w$ when $e_1$ is removed from $T$, and $v_2$ is adjacent to $w$ by an edge $e_2$ and $T_2$ is the component containing $w$ when $e_2$ is removed from $T$.  
\begin{figure}
\begin{center}
\begin{tikzpicture}[node distance = 2cm, v/.style = {circle, draw}]
\draw[dashed] (0, 0) ellipse [x radius = 3, y radius = 2];
\draw[dashed] (2, 0) ellipse [x radius = 3.3, y radius = 2];
\node[v] (w) {$w$};
\node[v] (v1) [left of = w] {$v_1$};
\node (blank) [right of = w] {};
\node[v] (v2) [right of = blank] {$v_2$};
\draw (v1) -- node [below] {$e_1$} (w) -- node [below] {$e_2$} (v2);
\draw (w) -- (1.5, -.75);
\draw (w) -- (.5, -1.25);
\draw (w) -- (1, 1) node [right, circle, draw] {$T''$};
\draw (v1) -- (-1.5, 1);
\draw (v1) -- (-1.5, -1);
\draw (v2) -- (3.5, -1);
\draw (v2) -- (3.5, 1);
\node [above of = v2] {$T_1$};
\node [above of = v1] {$T_2$};
\end{tikzpicture}
\end{center}
\caption{The form of tree~$T$ in the proof of Lemma~\ref{lemma:ell}(g)}
\label{fig:tree form}
\end{figure}
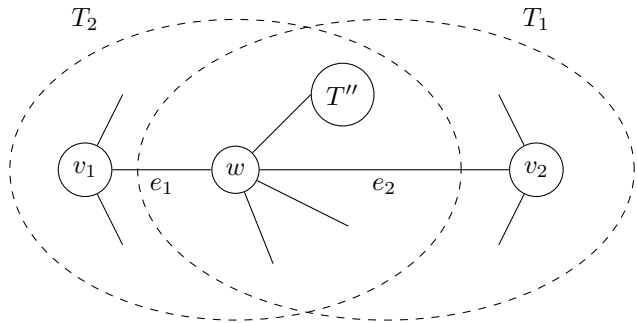
In this case, for any vertex $v$ of $T$, at most one of $e_1$ and $e_2$ lies on the path connecting $v$ to $w$, and therefore $V(T_1) \cup V(T_2) = V(T)$ and $w \in V(T_1) \cap V(T_2)$.  
However, this contradicts the hypotheses 
$\order{T} \geq 8$, $\order{T_1} \leq 4$, and $\order{T_2} \leq 4$.  
Thus $\phi$ is injective, as claimed.
\qedhere

\end{enumerate}
\end{proof}

In upcoming parts of the paper, 
the following type of tree will be useful.  
A \emph{spider} is a tree with exactly one vertex of degree bigger than~$2$.
(In the literature, spiders are also sometimes called \emph{star-like trees} or \emph{subdivisions of stars}.) 
The vertex of degree bigger than~$2$ is called the \emph{head} of the spider.
Let $T$ be a spider with head~$h$.
Each component of~$T \smallsetminus h$ is called a \emph{leg} of the spider.
Each leg is a path, and the number of legs of the spider is the degree of its head. 
The \emph{length} of a leg is its order, the number of vertices on it.

Given an integer~$d$ bigger than~$2$
and a list $\lambda_1$, \dots, $\lambda_d$ of positive integers, 
let $\spider_{\lambda_1, \dots, \lambda_d}$ denote the spider with a head of degree~$d$
and legs of length $\lambda_1$, \dots, $\lambda_d$.  
Thus the order of $\spider_{\lambda_1, \dots, \lambda_d}$ is $1 + \sum_i \lambda_i$. 
For example, 
the tree in Figure~\ref{fig:counter-example} is the spider~$\spider_{2, 2, 1}$, 
and the three trees in the first row of Figure~\ref{fig:all six vertices} are the spiders~$\spider_{3, 1, 1}$, $\spider_{2, 1, 1, 1}$, and~$\spider_{1, 1, 1, 1, 1}$.

The next lemma will use paths and spiders to characterize trees with two or three leaves.
We will use the following indicator notation:
$\indicator{\mathrm{true}} = 1$ and $\indicator{\mathrm{false}} = 0$.

\begin{lemma} \label{lemma:2-3 leaves} \
\begin{enumerate}[(a)]

\item 
If $T$ is a tree such that $\ell_1(T) = 2$, 
then $T$ is a path.

\item 
If $T$ is the path~$P_n$ and $1 \le k \le n - 1$,
then $\ell_k(P_n) = 2$.

\item  
If $T$ is a tree such that $\ell_1(T) = 3$, 
then $T$ is a spider with three legs.

\item  
If $T$ is the spider~$\spider_{\lambda_1, \dots, \lambda_d}$ 
and $1 \le k \le \sum_{i = 1}^d \lambda_i$,
then 
\[
  \ell_k(T) = \sum_{i = 1}^d \indicator{k \le \lambda_i} 
	              + \sum_{i = 1}^d \indicator{k \ge n - \lambda_i} \, .
\]	

\item
If $T$ is a tree such that $\ell_1(T) = 3$, $\ell_2(T) = 3$, and $\ell_3(T) = 0$, 
then $T$ is the spider~$\spider_{2, 2, 2}$, 
shown in Figure~\ref{fig:3-3-0}.
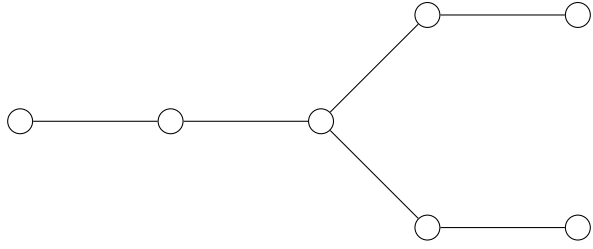
\begin{figure}
\begin{center}
\begin{tikzpicture}[node distance = 2cm, v/.style = {circle, draw}]
\node[v] (1)                {};
\node[v] (2) [right of = 1] {};
\node[v] (3) [right of = 2] {};
\draw (1) -- (2) -- (3);
\node[v] (4) [above right of = 3] {};
\node[v] (5) [right of = 4] {};
\draw (3) -- (4) -- (5);
\node[v] (6) [below right of = 3] {};
\node[v] (7) [right of = 6] {};
\draw (3) -- (6) -- (7);
\end{tikzpicture}
\end{center}
\caption{The unique tree for which $\ell_1 = 3$, $\ell_2 = 3$, and $\ell_3 = 0$}
\label{fig:3-3-0}
\end{figure}

\end{enumerate}
\end{lemma}

\begin{proof} \
\begin{enumerate}[(a)]

\item From Lemma~\ref{lemma:ell}(a), $\ell_1(T)$ is the number of leaves of $T$.  It is a standard exercise~\cite[Ex.~2.1.18]{West} that the number of leaves of a tree is at least its maximum degree $\Delta$, since $2(n - 1) = \sum_v d(v) \geq \ell_1(T) \cdot 1 + 1 \cdot \Delta + 2 \cdot (n - \ell_1(T) - 1)$.  Thus if $\ell_1(T) = 2$, the tree $T$ has maximum degree at most $2$, i.e., it is a path.

\item Labelling the path $v_1 \sim v_2 \sim \cdots \sim v_n$ and using the indexing from the proof of Lemma~\ref{lemma:ell}(c), we have for $k = 1, \ldots, n - 1$ that $L_k(T) = \{ (v_{k + 1}, T_{v_{k + 1}, v_k}), (v_{n - k}, T_{v_{n - k}, v_{n - k + 1}})\}$.

\item
If $T$ has three leaves, then it is not a path, and so some vertex has degree at least $3$.  By the argument in part (a), in fact the maximum degree must be exactly $3$.  Refining the argument in part (a), if we have equality $\Delta = \ell_1(T)$ then it must be the case that all vertices other than the leaves and a single vertex of degree $\Delta$ have degree $2$,  and so the graph is a spider with $\Delta$ legs \cite[Ex.~2.1.59(a)]{West}.

\item Deleting the vertices of the leg of length $\lambda_i$ contributes $\indicator{k < \lambda_i} + \indicator{k \geq n - \lambda_i}$, with the first term counting the contribution of the paths that do not contain the head and the second term counting the subtrees containing the head; deleting the head contributes $\sum_{i} \indicator{k = \lambda_i}$.  Summing up over all legs gives the result.

\item
Let $T$ be a tree such that $\ell_1(T) = \ell_2(T) = 3$ and $\ell_3(T) = 0$.  By part (c), since $\ell_1(T) = 3$, $T = \spider_{\lambda_1, \lambda_2, \lambda_3}$ is a spider with three legs.  Since $\ell_3(T) = 0$, we have by part (d) that $\lambda_i < 3$ (because the first sum must not have any positive terms).  Finally one easily checks by hand that among the remaining possibilities $\spider_{1, 1, 1}$, $\spider_{2, 1, 1}$, $\spider_{2, 2, 1}$, and $\spider_{2, 2, 2}$, only the last has $\ell_2(T) = 3$.
\qedhere

\end{enumerate}
\end{proof}

\section{Two players with optimal strategy}
\label{sec:strategic}

In this section, 
we consider the tree search game when both players are trying to maximize their probability of winning.  
The tree search game is a finite, two-player, zero-sum, sequential game,
and so both players have an optimal deterministic strategy.
For a tree~$T$, 
let $\PP(T)$ be the probability that the first player wins on~$T$
when both players play with an optimal strategy.
Our first theorem exactly computes this win probability for every tree.

\begin{theorem}
If $T$ is a tree on $n$ vertices, then
\[
\PP(T) = \frac{1}{2} + \frac{1}{2n} \cdot \indicator{n \textrm{ is odd}}   \, ,
\]
where $\indicator{n \textrm{ is odd}}$ denotes the indicator function for the event that $n$ is odd.
\end{theorem}
\begin{proof}
The result is easy to verify for $n = 1$.  Let $T$ be a tree on $n > 1$ vertices, and suppose the result is true for all trees on fewer than $n$~vertices.  For a vertex~$v$ of~$T$, let $\PP(T, v)$ be the probability that the first player wins on~$T$ if their first guess is~$v$, 
so $\PP(T) = \max_{v \in V(T)} \PP(T, v)$.  

With probability $\frac{1}{n}$, the first guess is correct.  Otherwise, the target vertex lies in one of the $k = \deg(v)$ components $T_{1}$, \dots, $T_{k}$ of the forest~$T \smallsetminus v$, and the probability that it lies in component~$T_i$ is~$\frac{\order{T_i}}{n}$.  Moreover, in this case, the second player wins with probability $\PP(T_i)$, and so\ the first player wins with probability $1 - \PP(T_i)$.  Thus
\[
\PP(T, v) 
  = \frac{1}{n} + \sum_{i = 1}^k \frac{|T_i|}{n} \cdot \bigl( 1 - \PP(T_i) \bigr) 
	= 1 - \frac{1}{n}\sum_{i = 1}^k |T_i| \cdot \PP(T_i).
\]
Let $m$ be the number of $T_i$ such that $|T_i|$ is odd.  
By the induction hypothesis, 
\[
\PP(T, v) 
  = 1 - \frac{1}{n} \Bigl( \frac{m}{2} + \sum_{i = 1}^k |T_i| \cdot \frac{1}{2} \Bigr) 
	= 1 - \frac{m}{2n} - \frac{n - 1}{2n} 
	= \frac{1}{2} + \frac{1 - m}{2n} \, .
\]
This quantity is maximized when $m$ is minimized.  If $n = |T|$ is even, then after removing $v$ there are an odd number of vertices and so $m \geq 1$; moreover, we can always achieve the minimum $m = 1$ by taking $v$ to be a leaf (among possibly other options), and so $\PP(T) = \frac{1}{2}$ in this case.  If $n$ is odd, then we can always achieve the minimum possible value $m = 0$ by taking $v$ to be a leaf (among possibly other options), and so $\PP(T) = \frac{1}{2} + \frac{1}{2n}$ in this case.  By induction, the result is valid for all trees.
\end{proof}

\begin{remark}
It follows from the preceding proof that the set of optimal moves for the first player is precisely the set of vertices $v$ such that $T \smallsetminus v$ has the minimum number of odd-order components (namely, $0$ if $n$ is odd and $1$ if $n$ is even).  In particular, it is always optimal to choose a leaf.  In the motivating context of binary search on $\{1, \dots, n\}$, we have that all first moves are equally strong if $n$ is even, while the first player should guess any odd number when $n$ is odd.
\end{remark}

\section{Optimal strategy versus a random player}
\label{sec:mixed}

In this section,
we consider the tree search game when 
one player is trying to maximize their probability of winning
while the other plays uniformly at random.
This version of the game is effectively a one-player game,
since only one player is playing strategically.
For a tree~$T$,
let $\PP(T)$ be the probability that a player playing with an optimal strategy wins 
when they have the first move on~$T$ 
against an opponent who chooses vertices uniformly at random.
Similarly, 
let $\QQ(T)$ be the probability that a player playing with an optimal strategy wins 
when they have the second move on~$T$ 
against an opponent who chooses vertices uniformly at random.  

It is easy to compute that 
$\PP(S_1) = 1$, $\QQ(S_1) = 0$, and $\PP(S_2) = \QQ(S_2) = \frac{1}{2}$.  
For the tree on three vertices, 
a player wins with probability~$\frac{1}{3}$ if they choose the middle node 
(if the choice is incorrect, then the opponent has only one vertex to choose from) 
and probability $\frac{1}{3} + \frac{2}{3} \cdot \frac{1}{2} = \frac{2}{3}$ 
if they choose a leaf, 
so $\PP(S_3) = \frac{2}{3}$ and~$\QQ(S_3) = 1 - \frac{1}{3} \left(\frac{2}{3} + \frac{1}{3} + \frac{2}{3}\right) = \frac{4}{9}$.

For a vertex~$v$ of tree~$T$, 
let $\PP(T, v)$ be the probability that the optimal player wins when playing first on~$T$, 
provided that the optimal player's first move is~$v$.
Since the first player plays optimally, 
\[
  \PP(T) = \max_{v \in V(T)} \PP(T, v) .
\]
Since the target vertex is selected uniformly at random,
\[
  \PP(T, v) 
	  = \frac{1}{n} + \sum_{T'} \frac{\order{T'}}{n} \QQ(T') 
		= \frac{1}{n} + \frac{1}{n} \sum_{T'} \order{T'} \QQ(T') ,
\]
where $T'$ ranges over the components of the forest~$T \smallsetminus v$.
Similarly, 
let $\QQ(T, v)$ be the conditional probability that the optimal player wins 
when playing second on~$T$, 
given that the random player's first move is~$v$.
Since the first move is chosen uniformly at random, 
\[
  \QQ(T) = \frac{1}{n} \sum_{v \in V(T)} \QQ(T, v) .
\]
Moreover, since the target vertex is selected uniformly at random,
\[
  \QQ(T, v)  
	  = \sum_{T'} \frac{\order{T'}}{n} \PP(T') 
		= \frac{1}{n} \sum_{T'} \order{T'} \PP(T') ,
\]
where $T'$ ranges over the components of~$T \smallsetminus v$.
Combining the last two formulas gives
\[
  \QQ(T) = \frac{1}{n^2} \sum_{v \in V(T)} \sum_{T'} \order{T'} \PP(T') .
\]

In the following two subsections, 
we give exact values for the win probabilities $\PP$ and $\QQ$ 
for star graphs and for path graphs,
respectively.  
In a third subsection, we give bounds for general trees. 

\subsection{Stars}

The next theorem exactly computes the probabilities of winning for every star graph.

\begin{theorem}  \label{thm:semi-random stars}
For every positive integer $n$,
\begin{align*}
\PP(S_n) & = \frac{2}{3} - \frac{1}{3n} +  \frac{2}{3n} \cdot \frac{\binom{n - 1}{(n - 1)/2}}{2^{n - 1}} \cdot \indicator{n \textrm{ is odd}}\\
\intertext{and}
\QQ(S_n) & = \frac{2}{3} - \frac{2}{3n} +  \frac{2}{3n} \cdot \frac{\binom{n }{n/2}}{2^n} \cdot \indicator{n \textrm{ is even}}  \, .
\end{align*}
\end{theorem}
\begin{proof}
For $n = 1$ the proposed formulas give $\PP(S_1) = \frac{2}{3} - \frac{1}{3} + \frac{2}{3} = 1$ and $\QQ(S_1) = \frac{2}{3} - \frac{2}{3} = 0$, and for $n = 2$ they give $\PP(S_2) = \frac{2}{3} - \frac{1}{6} = \frac{1}{2}$ and $\QQ(S_2) = \frac{2}{3} - \frac{1}{3} + \frac{1}{3} \cdot \frac{1}{2} = \frac{1}{2}$, as needed.

For $n \geq 3$, 
the strategic player may choose either a leaf~$\ell$ or the center vertex~$c$.  
One has
\[
  \PP(S_n, \ell) 
	  = \frac{1}{n}\cdot 1 + \frac{n - 1}{n} \cdot \QQ(S_{n - 1}) 
		> \frac{1}{n} 
		= \PP(S_n, c),
\]
so it is always optimal to choose a leaf and
\[
  \PP(S_n) = \frac{1}{n} + \frac{n - 1}{n} \QQ(S_{n - 1}).
\]
Similarly, 
considering separately whether the opponent playing randomly chooses the center or a leaf, one has
\[
  \QQ(S_n) 
	  = \underbrace{\frac{1}{n} \cdot \frac{n - 1}{n} \cdot 1}_{\textrm{chooses center}} 
		    + \underbrace{\frac{n - 1}{n} \cdot \frac{n - 1}{n} \cdot 
				    \PP(S_{n - 1})}_{\textrm{chooses leaf}}.
\]
It is straightforward to verify that the given formulas satisfy these recurrence equations, so the result holds by induction.
\end{proof}

One may immediately compute the asymptotic win probabilities for stars.

\begin{corollary}
One has 
\[
  \lim_{n \to \infty} \PP(S_n) = \lim_{n \to \infty} \QQ(S_n) = \frac{2}{3} \, .
\]
\end{corollary}

\subsection{Paths}

In this subsection,
we exactly compute the win probabilities $\PP$ and $\QQ$ for every path graph.
Given a positive integer~$n$, define~$p_n$ and~$q_n$ by
\begin{align*}
  p_n &= \frac{1}{2} - \frac{1}{n} \cdot \frac{(-2)^n}{n!}
  		     + \frac{n + 2}{2n}  \sum_{j = 0}^{n} \frac{(-2)^j}{j!} \, ,  \\
	q_n &= \frac{n-1}{2n} - \frac{1}{n} \cdot \frac{(-2)^n}{n!}
	         + \frac{n + 3}{2n} \sum_{j = 0}^{n} \frac{(-2)^j}{j!} \, .
\end{align*}
We will show that $\PP(P_n) = p_n$ is 
the probability that the player playing strategically beats 
an opponent playing randomly on a path with $n$~vertices, 
provided that the optimal player goes first.
Similarly, we will show that $\QQ(P_n) = q_n$ is 
the probability that an optimal player beats a random opponent on a path with $n$~vertices, 
provided that the optimal player goes second.

Because the infinite series $\sum_{j = 0}^\infty \frac{(-2)^j}{j!}$ converges to $e^{-2}$,
the limits (as $n$ approaches infinity) of $p_n$ and $q_n$ are both $(1 + e^{-2}) / 2$.
For later use, we record the first seven values of $p_n$ and $q_n$ in Table~\ref{table:p's and q's}.
\begin{table}
\[
\begin{array}{c|ccccccc}
    n & 1 &   2 &   3 &    4 &     5 &     6 &       7  \\
  p_n & 1 & 1/2 & 2/3 & 7/12 &   3/5 & 53/90 &   37/63  \\
  q_n & 0 & 1/2 & 4/9 &  1/2 & 38/75 & 14/27 & 386/735        
\end{array}
\]
\caption{The first few values of $p_n$ and $q_n$}
\label{table:p's and q's}
\end{table}
For convenience, define $p_0$ and $q_0$ to be~$\frac{1}{2}$.

The following lemma shows how to express~$p_n$ in terms of~$q_{n - 1}$.

\begin{lemma} \label{lemma:path endpoint}
If $n$ is a positive integer, then
\[
  p_n = \frac{1}{n} + \frac{n - 1}{n} q_{n - 1} \, .
\]
\end{lemma}

\begin{proof}
The case~$n = 1$ is easy, so we may assume that~$n \ge 2$.  
Using the definitions of $p$ and~$q$, we have
\begin{align*}
  n p_n 
	  &= \frac{n}{2} - \frac{(-2)^n}{n!}  + \frac{n + 2}{2}  \sum_{j = 0}^{n} \frac{(-2)^j}{j!}  \\
%		&= \frac{n}{2} - \frac{(-2)^n}{n!} + \frac{n + 2}{2} \cdot \frac{(-2)^n}{n!}
%			+ \frac{n + 2}{2}  \sum_{j = 0}^{n - 1} \frac{(-2)^j}{j!} \\
		&= \frac{n}{2} + \frac{n}{2} \cdot \frac{(-2)^n}{n!} 
			+ \frac{n + 2}{2}  \sum_{j = 0}^{n - 1} \frac{(-2)^j}{j!}  \\		
		&= 1 + \frac{n - 2}{2} - \frac{(-2)^{n - 1}}{(n - 1)!} 
			+ \frac{n + 2}{2}  \sum_{j = 0}^{n - 1} \frac{(-2)^j}{j!} \\
		&= 1 + (n - 1) q_{n - 1} \, ,		
\end{align*}
as desired. 
\end{proof} 

Our next lemma expresses~$q_n$ as an average involving previous values of~$p$.

\begin{lemma}  \label{lemma:path average}
If $n$ is a positive integer, then
\[
  q_n 
	  = \frac{1}{n} \sum_{k = 1}^n \Bigl( \frac{k - 1}{n} p_{k - 1} + \frac{n - k}{n} p_{n - k} \Bigr) .
\]
\end{lemma}

\begin{proof}
The case $n = 1$ is easy.  
For $n \geq 2$, we can simplify the desired equation as follows:
\[
  n^2 q_n = 2 \sum_{k = 1}^{n - 1} k p_k \, .
\]
By telescoping sums, it is sufficient to verify the difference equation
\[
  (k + 1)^2 q_{k + 1} - k^2 q_k = 2 k p_k \, .
\]
Plugging in the definition of $q_k$ and $q_{k + 1}$ 
and then using the definition of~$p_k$, we have
\begin{align*}
  (k + 1)^2 q_{k + 1} - k^2 q_k 
    &= k + \frac{k + 4}{2} \cdot \frac{(-2)^{k + 1}}{k!} - \frac{(-2)^{k + 1}}{k!}
		+ k \cdot \frac{(-2)^k}{k!} + (k + 2) \sum_{j = 0}^k \frac{(-2)^j}{j!} \\
	&= k  - 2 \cdot \frac{(-2)^k}{k!} + (k + 2) \sum_{j = 0}^k \frac{(-2)^j}{j!} \\
	&= 2 k p_k \, ,			
\end{align*}
which is the difference equation.
\end{proof} 

The next lemma gives good upper and lower bounds on~$q_n$.  

\begin{lemma}  \label{lemma:path estimate}
If $n$ is a positive integer, then
\[
  n q_n 
		\le \frac{1 + e^{-2}}{2} n - \frac{1}{8} \, .
\]
Furthermore, if $n \ge 3$, then
\[
  \frac{1 + e^{-2}}{2} n - \frac{3}{8} 
	  \le n q_n 
		\le \frac{1 + e^{-2}}{2} n - \frac{1}{4} \, .
\]
\end{lemma}

\begin{proof}
Using Table~\ref{table:p's and q's},
we can check the cases $n \le 7$, 
so we may assume that $n \ge 8$.
By Taylor's theorem applied to the exponential function, 
\[
  \abs[\Big]{\sum_{j = 0}^{n} \frac{(-2)^j}{j!} - e^{-2}} 
	  \le \frac{2^{n + 1}}{(n + 1)!} \, .  
\]
Plugging this bound into the definition of~$q_n$ gives
\[
  \abs[\Big]{n q_n - \frac{1 + e^{-2}}{2} n + \frac{1 - 3 e^{-2}}{2}} 
	  \le \frac{n + 3}{2} \cdot \frac{2^{n + 1}}{(n + 1)!} + \frac{2^n}{n!}
		= (n + 2) \frac{2^{n + 1}}{(n + 1)!} \, .
\] 
Since $n \ge 8$, 
the right side $(n + 2) 2^{n + 1} / (n + 1)!$ is at most~$\frac{1}{25}$.
Hence 
\[
  \frac{1 + e^{-2}}{2} n - \frac{1 - 3 e^{-2}}{2} - \frac{1}{25}
	  \le n q_n
		\le \frac{1 + e^{-2}}{2} n - \frac{1 - 3 e^{-2}}{2} + \frac{1}{25} \, .
\] 
Since $(1 - 3e^{-2}) / 2$ is between $0.29$ and~$0.3$, we are done.
\end{proof}

Since $q_2 > q_3$,
the sequence~$\{ q_n \}_{n \ge 1}$ is not increasing,
though the sequence~$\{ q_n \}_{n \ge 3}$ is strictly increasing.  
Our next lemma shows that the sequence~$\{ q_n \}_{n \ge 1}$ satisfies an increasing-like property,
namely, the sequence~$\{ n q_n \}_{n \ge 1}$ is superadditive.

\begin{lemma}  \label{lemma:superadditive paths}
If $m$ and $n$ are nonnegative integers, then
\[
  m q_m + n q_n \le (m + n) q_{m + n} \, .
\]
\end{lemma}

\begin{proof}
By symmetry, we may assume that~$m \le n$.
The case~$m = 0$ is trivial, so we may assume that~$m \ge 1$.
The cases~$n \le 2$ are verified in Table~\ref{table:small cases}.  
Hence we may further assume that~$n \ge 3$. 
\begin{table}
\[
\begin{array}{cccc}
  m & n & m q_m + n q_n & (m + n) q_{m + n} \\ \hline
	1 & 1 & 0             & 1                 \\
	1 & 2 & 1             & 4/3               \\
	2 & 2 & 2             & 2            
\end{array}
\]
\caption{Small cases in the proof of Lemma~\ref{lemma:superadditive paths}}
\label{table:small cases}
\end{table}
Using Lemma~\ref{lemma:path estimate} twice, we have
\[
  m q_m \le \frac{1 + e^{-2}}{2} m - \frac{1}{8}  \qquad \textrm{ and } \qquad
	n q_n \le \frac{1 + e^{-2}}{2} n - \frac{1}{4} \, .
\]
Adding these two inequalities and using Lemma~\ref{lemma:path estimate} again, we have
\[
  m q_m + n q_n 
		\le \frac{1 + e^{-2}}{2} (m + n) - \frac{3}{8}
		\le (m + n) q_{m + n} \, .  \qedhere
\] 
\end{proof}

Lemma~\ref{lemma:path endpoint} expressed~$p_n$ in terms of~$q_{n - 1}$.
The next lemma shows that $p_n$ can be expressed as a maximum involving all previous~$q$.

\begin{lemma}  \label{lemma:path max}
If $n$ is a positive integer, then
\[
  p_n 
	  = \max_{1 \le k \le n} 
		    \Bigl( \frac{1}{n} + \frac{k - 1}{n} q_{k - 1} + \frac{n - k}{n} q_{n - k} \Bigr) .
\]
\end{lemma}

\begin{proof}
By considering $k = 1$ and using Lemma~\ref{lemma:path endpoint}, 
we see that the maximum is at least~$p_n$.
On the other hand, for every~$k$, 
by Lemmas~\ref{lemma:path endpoint} and~\ref{lemma:superadditive paths}, 
we have
\[
  \frac{1}{n} + \frac{k - 1}{n} q_{k - 1} + \frac{n - k}{n} q_{n - k}
	  \le \frac{1}{n} + \frac{n - 1}{n} q_{n - 1} = p_n \, .
\]
Hence the maximum is at most~$p_n$.
\end{proof}

Finally, 
we are ready to prove our formulas for the probabilities of winning on a path.  

\begin{theorem}
If $n$ is a positive integer, 
then $\PP(P_n) = p_n$ and $\QQ(P_n) = q_n$.
\end{theorem}

\begin{proof}
We may assume that the vertices of the path in order are $1$, $2$, \dots,~$n$.
Recall that $\PP(P_n)$ is $\max_k \PP(P_n, k)$.
For every vertex~$k$, we have
\[
  \PP(P_n, k) = \frac{1}{n} + \frac{k - 1}{n} \QQ(P_{k - 1}) + \frac{n - k}{n} \QQ(P_{n - k}).
\]
Hence 
Lemma~\ref{lemma:path max} provides the correct recurrence relation for~$\PP(P_n)$:
it is given by the maximum over all vertices~$k$ from $1$ to~$n$.
Similarly, 
recall that $\QQ(P_n)$ is the average of $\QQ(P_n, k)$ over all vertices~$k$.
For every vertex~$k$, we have
\[
  \QQ(P_n, k) = \frac{k - 1}{n} \PP(P_{k - 1}) + \frac{n - k}{n} \PP(P_{n - k}).
\]
Hence 
Lemma~\ref{lemma:path average} provides the correct recurrence relation for~$\QQ(P_n)$: 
it is given by the average over all vertices~$k$ from $1$ to~$n$.
The theorem follows by induction on~$n$.
\end{proof}

\subsection{General trees} 

In this subsection,
we bound the probabilities $\PP(T)$ and $\QQ(T)$ of winning for every tree.
We start with an intriguing conjecture that paths and stars 
have the extreme win probabilities.

\begin{conjecture}
\label{conj:half random tree bounds}
For every tree~$T$ on $n$~vertices,
\[
\PP(P_n) \leq \PP(T) \leq \PP(S_n)  \qquad \textrm{ and } \qquad 
\QQ(P_n) \leq \QQ(T) \leq \QQ(S_n).
\]
\end{conjecture}

We have confirmed the conjecture for~$n \le 20$.  We are able to prove half of the conjecture: stars have the largest probabilities of winning.

\begin{theorem}  \label{thm:semi-random upper bound}
If $T$ is a tree with $n$~vertices, then
\[
  \PP(T) \le \PP(S_n)
    \qquad \text{ and } \qquad
  \QQ(T) \le \QQ(S_n) .
\]
\end{theorem}

\begin{remark}
Before embarking on the proof, 
we observe one complication that may help explain why this semi-random version is more complicated than the fully optimal play.  
Unlike the situation in which both players play optimally, 
it is \emph{not} the case that there is always a leaf among the optimal moves.  
Consider the tree~$T$ formed by starting from the path~$P_5$ and adding two edges incident to each endpoint; 
see Figure~\ref{fig:suboptimal leaves}.
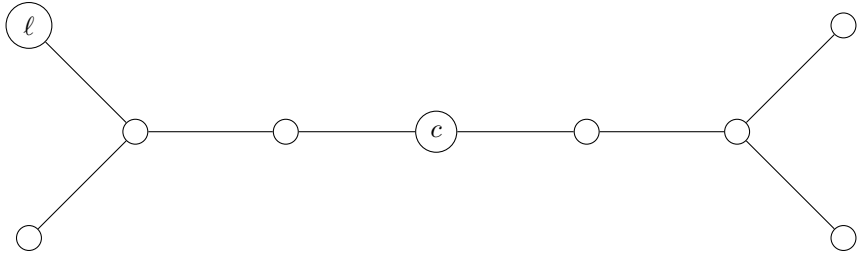
\begin{figure}
\begin{center}
\begin{tikzpicture}[node distance = 2cm, v/.style={circle,draw}]
\node[v] (1) {$\ell$};
\node[v] (2) [below right of = 1] {};
\node[v] (3) [below left of = 2] {};
\node[v] (4) [right of = 2] {};
\node[v] (5) [right of = 4] {$c$};
\node[v] (6) [right of = 5] {};
\node[v] (7) [right of = 6] {};
\node[v] (8) [above right of = 7] {};
\node[v] (9) [below right of = 7] {};
\draw (1) -- (2);
\draw (3) -- (2) -- (4) -- (5) -- (6) -- (7) -- (8);
\draw (7) -- (9);
\end{tikzpicture}
\end{center}
\caption{A tree for which every leaf is a suboptimal move}
\label{fig:suboptimal leaves}
\end{figure}
Then for the central vertex~$c$ one has
\[
  \PP(T, c)     % = \frac{1}{9} + \frac{8}{9} \QQ(S_4)
    = \frac{1}{9} + \frac{8}{9} \cdot \frac{9}{16} \approx 0.611 ,
\]
while for a leaf~$\ell$ one has
\[
  \PP(T, \ell)  % = \frac{1}{9} + \frac{8}{9} \QQ(D_8) = 
    = \frac{1}{9} + \frac{8}{9} \cdot \frac{12601}{23040}  \approx 0.597 .
\]
\end{remark}

As a first step toward the proof of Theorem~\ref{thm:semi-random upper bound}, 
we study the sequence~$\{ \QQ(S_n) \}_{n \ge 1}$.  
This sequence alternately increases and decreases, 
but we prove a loose monotonicity property on the values 
as they drift upward toward~$\frac{2}{3}$.
Namely, 
the next lemma shows that the sequence~$\{ n \QQ(S_n) \}_{n \ge 1}$ is superadditive.

\begin{lemma}  \label{lemma:superadditive stars}
If $m$ and $n$ are positive integers, then
\[
  m \QQ(S_m) + n \QQ(S_n) \le (m + n) \QQ(S_{m + n}) .
\]
\end{lemma}

\begin{proof}
If $k$ is a positive even number, 
then $\binom{k}{k/2} / 2^k$ is at most~$\frac{1}{2}$.
Hence, by using Theorem~\ref{thm:semi-random stars} twice, we have
%\begin{align*}
%  m \QQ(S_m) 
%	  &\le \frac{2}{3} m - \frac{2}{3} + \frac{2}{3} \cdot \frac{1}{2} = \frac{2}{3} m - \frac{1}{3} \\
%  n \QQ(S_n) 
%	  &\le \frac{2}{3} n - \frac{2}{3} + \frac{2}{3} \cdot \frac{1}{2} = \frac{2}{3} n - \frac{1}{3} \, .
%\end{align*}
\[
  m \QQ(S_m) 
	  \le \frac{2}{3} m - \frac{2}{3} + \frac{2}{3} \cdot \frac{1}{2} = \frac{2}{3} m - \frac{1}{3} \qquad \textrm{ and } \qquad
  n \QQ(S_n) 
	  \le \frac{2}{3} n - \frac{2}{3} + \frac{2}{3} \cdot \frac{1}{2} = \frac{2}{3} n - \frac{1}{3} \, .
\]
Adding these two inequalities and using Theorem~\ref{thm:semi-random stars} again, we conclude
\[
  m \QQ(S_m) + n \QQ(S_n) 
		\le  \frac{2}{3} m - \frac{1}{3} + \frac{2}{3} n - \frac{1}{3}
		= \frac{2}{3} (m + n) - \frac{2}{3}
		\le (m + n) \QQ(S_{m + n}) .  \qedhere
\] 
\end{proof}

\begin{proof}[Proof of Theorem~\ref{thm:semi-random upper bound}]
The proof is by strong induction on~$n$.
The base case~$n = 1$ is easy.  
For $n \ge 2$,
assume the result is true for every tree with fewer than~$n$ vertices.
We will prove the result for tree~$T$.

First we prove the bound on $\PP(T)$.
Let $v$ be a vertex of~$T$.
We have
\[
  \PP(T, v)
	  = \frac{1}{n} + \sum_{T'} \frac{\order{T'}}{n} \QQ(T')
		= \frac{1}{n} + \frac{1}{n} \sum_{T'} \order{T'} \QQ(T') ,
\]
where $T'$ ranges over the components of~$T \smallsetminus v$.
By the induction hypothesis and Lemma~\ref{lemma:superadditive stars}, 
\[
  \PP(T, v)
	  \le \frac{1}{n} + \frac{1}{n} \sum_{T'} \order{T'} \QQ(S_{\order{T'}}) 
		\le \frac{1}{n} + \frac{n - 1}{n} \QQ(S_{n - 1}) .
\]
As shown in the proof of Theorem~\ref{thm:semi-random stars},
the right side equals~$\PP(S_n)$, so $\PP(T, v) \le \PP(S_n)$.
We conclude that
\[
  \PP(T) 
	  = \max_{v \in V(T)} \PP(T, v) 
		\le \PP(S_n) .
\]

Next we prove the bound on $\QQ(T)$.
Let $v$ be a vertex of~$T$.
We have
\[
  \QQ(T, v)
	  = \sum_{T'} \frac{\order{T'}}{n} \PP(T')
		= \frac{1}{n} \sum_{T'} \order{T'} \PP(T') ,
\]
where $T'$ ranges over the components of~$T \smallsetminus v$.
By the induction hypothesis,
\[
  \QQ(T, v)
	  \le \frac{1}{n} \sum_{T'} \order{T'} \PP(S_{\order{T'}}) .
\]
Given a positive integer~$m$, let 
$C(m)$ be $\binom{m - 1}{(m - 1)/2} / 2^{m - 1}$ if $m$ is odd and $0$ otherwise.
Theorem~\ref{thm:semi-random stars} says that
\[
  m \PP(S_m) 
	  = \frac{2}{3} m  - \frac{1}{3} + \frac{2}{3} C(m) \, .
\]
Plugging this formula into our inequality for~$\QQ(T, v)$ gives 
\[
  \QQ(T, v)
	  \le \frac{1}{n} \sum_{T'} \Bigl( \frac{2}{3} \order{T'} - \frac{1}{3} + \frac{2}{3} C(\order{T'}) \Bigr)
		= \frac{2(n - 1)}{3n} - \frac{1}{3n} \deg(v) + \frac{2}{3n} \sum_{T'} C(\order{T'}) .
\]
Given a vertex~$w$ that is a neighbor of~$v$,
recall from the proof of Lemma~\ref{lemma:ell}(c) that the tree~$T_{v, w}$ is 
the component of $T \smallsetminus v$ that contains~$w$.
We can rewrite our previous inequality as
\[
  \QQ(T, v)
		\le \frac{2(n - 1)}{3n} - \frac{1}{3n} \deg(v) 
		      + \frac{2}{3n} \sum_{\substack{w \in V(T) \\ w \sim v}} C(\order{T_{v, w}}) .
\]
Averaging over all vertices~$v$ gives the bound
\[
  \QQ(T)
	  = \frac{1}{n} \sum_{v \in V(T)} \QQ(T, v) 
		\le \frac{2(n - 1)^2}{3n^2} 
		      + \frac{2}{3n^2} \sum_{\substack{v, w \in V(T) \\ v \sim w}} C(\order{T_{v, w}}) .
\]

Given an edge~$\{v, w\}$ of~$T$,
we claim that $C(\order{T_{v, w}}) + C(\order{T_{w, v}})$ is at most~$1 + C(n - 1)$. 
Recall that $\order{T_{v, w}} + \order{T_{w, v}}$ is $n$.
If either $\order{T_{v, w}}$ or $\order{T_{w, v}}$ is~$1$,
then the other is~$n - 1$, so the claim is true with equality.  
Otherwise, both $\order{T_{v, w}}$ and $\order{T_{w, v}}$ are greater than~$1$,
in which case $C(\order{T_{v, w}})$ and $C(\order{T_{w, v}})$ are each at most~$\frac{1}{2}$,
so the claim is again true.
Plugging our claim into our inequality for~$\QQ(T)$  gives
\[
  \QQ(T)
	  \le \frac{2(n - 1)^2}{3n^2} + \frac{2(n - 1)}{3n^2} \bigl( 1 + C(n - 1) \bigr)
		= \frac{2}{3} - \frac{2}{3n} + \frac{2(n - 1)}{3n^2} C(n - 1) .
\]	
It is straightforward to verify that $(n - 1) C(n - 1) = n C(n + 1)$, so
\[
  \QQ(T)
	  \le \frac{2}{3} - \frac{2}{3n} + \frac{2}{3n} C(n + 1) .
\]
By Theorem~\ref{thm:semi-random stars},
the right side is $\QQ(S_n)$, which completes the proof.
\end{proof}

The next theorem makes progress toward the lower bounds of 
Conjecture~\ref{conj:half random tree bounds}.
In particular, 
we show that the player playing optimally wins 
with probability at least~$\frac{9}{16}$ (asymptotically).

\begin{theorem}
Let $T$ be a tree on $n$ vertices.
If $n \ge 4$, then
\[
  \PP(T) > \frac{9}{16}  \qquad \textrm{ and } \qquad
	\QQ(T) \geq \frac{9}{16} - \frac{5}{16n} \, .
\]
\end{theorem}

\begin{proof}
Define the function~$\delta$ by
$\delta(1) = \frac{5}{16}$, $\delta(2) = - \frac{1}{4}$, $\delta(3) = \frac{3}{16}$,
$\delta(4) = - \frac{1}{24}$, $\delta(5) = \frac{1}{16}$, and $\delta(k) = 0$ for~$k \ge 6$.
We will prove, for every positive integer~$n$, the more precise inequalities
\[
  \PP(T) \geq \frac{9}{16} + \frac{1}{8n} + \frac{\delta(n)}{n}  \qquad \textrm{ and } \qquad
	\QQ(T) \geq \frac{9}{16} - \frac{5}{16n} + \frac{\delta(n + 1)}{n} \, .
\]
The proof is by strong induction on~$n$.
The base cases~$n \le 4$ are easy to check.  
For $n \ge 5$, assume the inequalities hold for every tree with fewer than $n$~vertices.
We will prove the inequalities for tree~$T$.

First we bound~$\PP(T)$.
Let $v$ be a leaf of~$T$.
By the induction hypothesis, 
\begin{align*}
\PP(T) 
  &\ge \PP(T, v) \\
  &= \frac{1}{n} + \frac{n - 1}{n} \QQ(T \smallsetminus v) \\
  &\ge \frac{1}{n} + \frac{n - 1}{n} 
	                     \left( \frac{9}{16} - \frac{5}{16(n - 1)} + \frac{\delta(n)}{n - 1} \right) \\
  &= \frac{9}{16} + \frac{1}{n} - \frac{9}{16n} - \frac{5}{16n} + \frac{\delta(n)}{n} \\
  &= \frac{9}{16} + \frac{1}{8n} + \frac{\delta(n)}{n} \, ,
 \end{align*}
as desired.  

Next we bound~$\QQ(T)$.
By the recurrence formula for~$\QQ$, we have
\begin{align*}
\QQ(T)
  &= \frac{1}{n^2} \sum_v \sum_{T'} \order{T'} \PP(T') \\
	&\ge \frac{1}{n^2} 
	        \sum_v \sum_{T'} \left( \frac{9}{16} \order{T'} + \frac{1}{8} + \delta(\order{T'}) \right) \\
	&= \frac{9(n - 1)}{16 n} + \frac{n - 1}{4 n^2} 
	     + \frac{1}{n^2} \sum_v \sum_{T'} \delta(\order{T'}) \\
	&= \frac{9}{16} - \frac{5}{16n}	- \frac{1}{4n^2} 
		   + \frac{1}{n^2} 
			      \left( \frac{5}{16} \ell_1(T) - \frac{1}{4} \ell_2(T) + \frac{3}{16} \ell_3(T) 
						         - \frac{1}{24} \ell_4(T) + \frac{1}{16} \ell_5(T) \right) \\
	&= \frac{9}{16} - \frac{5}{16n} + \frac{1}{48 n^2}
						\Bigl( 15 \ell_1(T) - 12 \ell_2(T) + 9 \ell_3(T) - 2 \ell_4(T) + 3 \ell_5(T) - 12 \Bigr) .
\end{align*}

To prove the desired bound on~$\QQ(T)$, 
it suffices to prove that the expression in parentheses is nonnegative.  
For short, write $\ell_k$ for~$\ell_k(T)$. 
If $T$ is the unique exceptional tree from Lemma~\ref{lemma:ell}(g) 
(shown in Figure~\ref{fig:counter-example}),
then $\ell_1 = 3$, $\ell_2 = 2$, $\ell_3 = 0$, $\ell_4 = 2$, and $\ell_5 = 3$, which means
\[
  15 \ell_1 - 12 \ell_2 + 9 \ell_3 - 2 \ell_4 + 3 \ell_5
	  = 45 - 24 + 0 - 4 + 9
		= 26.
\]
Hence we may assume that $T$ is not the exceptional tree.
By Lemma~\ref{lemma:ell}(g), we have $\ell_1 + \ell_3 \ge \ell_2 + \ell_4$, so
\[
  15 \ell_1 - 12 \ell_2 + 9 \ell_3 - 2 \ell_4 + 3 \ell_5
	  \ge 13 \ell_1 - 10 \ell_2 + 7 \ell_3 + 3 \ell_5.
\]
If $\ell_1 = 2$, 
then $T$ is a path (Lemma~\ref{lemma:2-3 leaves}(a)), 
so $\ell_2 = \ell_3 = 2$ (Lemma~\ref{lemma:2-3 leaves}(b)), which means
\[
  13 \ell_1 - 10 \ell_2 + 7 \ell_3 + 3 \ell_5
	  \ge 13 \ell_1 - 10 \ell_2 + 7 \ell_3 
		= 26 - 20 + 14
		= 20.
\]
If $\ell_1 \ge 4$, then because $\ell_1 \ge \ell_2$ (Lemma~\ref{lemma:ell}(f)), we have
\[
  13 \ell_1 - 10 \ell_2 + 7 \ell_3 + 3 \ell_5
	  \ge 13 \ell_1 - 10 \ell_2 
		\ge 3 \ell_1
		\ge 12.
\]
Hence we may assume that $\ell_1 = 3$.
If $\ell_2 \le 2$, then
\[
  13 \ell_1 - 10 \ell_2 + 7 \ell_3 + 3 \ell_5
	  \ge 13 \ell_1 - 10 \ell_2 
	  = 39 - 10 \ell_2
		\ge 39 - 20
		= 19.
\]
Hence we may assume that $\ell_2 = 3$.
If $\ell_3 \ge 1$, then
\[
  13 \ell_1 - 10 \ell_2 + 7 \ell_3 + 3 \ell_5
	  \ge 13 \ell_1 - 10 \ell_2 + 7 \ell_3 
		= 39 - 30 + 7 \ell_3
		\ge 39 - 30 + 7
		= 16.
\]
Hence we may assume that $\ell_3 = 0$.
By Lemma~\ref{lemma:2-3 leaves}(e), 
the only tree with $\ell_1 = 3$, $\ell_2 = 3$, and $\ell_3 = 0$ is the spider~$\spider_{2, 2, 2}$, 
shown in Figure~\ref{fig:3-3-0}.
For this tree, $\ell_5 = 3$, which means
\[
  13 \ell_1 - 10 \ell_2 + 7 \ell_3 + 3 \ell_5
	  = 39 - 30 + 0 + 9
		= 18.  
\]
In every case, we have $15 \ell_1 - 12 \ell_2 + 9 \ell_3 - 2 \ell_4 + 3 \ell_5 - 12 \geq 0$.
This completes the proof.
\end{proof}

\section{Two random players}
\label{sec:random}

In this section,
we consider the tree search game when both players play randomly.
Given a tree~$T$, 
let $\PP(T)$ be the probability that the first player wins on~$T$ when both players choose vertices uniformly at random from among the vertices that could be the target.
For example, $\PP(S_1) = 1$, $\PP(S_2) = \frac{1}{2}$, and $\PP(S_3) = \frac{5}{9}$.

For every vertex~$v$ of a tree~$T$, 
let $\PP(T, v)$ be the conditional probability that the first player wins on~$T$
given that they select~$v$ first.
Since the first contestant plays uniformly at random, 
\[
  \PP(T) = \frac{1}{n} \sum_{v \in V(T)} \PP(T, v) .
\]
Moreover, since the target vertex is selected uniformly at random,
\[
  \PP(T, v) 
	  = \frac{1}{n} + \sum_{T'} \frac{\order{T'}}{n} \bigl( 1 - \PP(T') \bigr)
		= 1 - \frac{1}{n} \sum_{T'} \order{T'} \PP(T') ,
\]
where $T'$ ranges over the components of $T \smallsetminus v$.
Combining these two formulas gives
\[
  \PP(T) = 1 - \frac{1}{n^2} \sum_{v \in V(T)} \sum_{T'} \order{T'} \PP(T') .
\]

We first exactly compute the probability of winning for every star.

\begin{theorem}  \label{thm:random stars}
If $n$ is a positive integer, then
\[
  \PP(S_n) = \frac{1}{2} + \frac{1}{2n^2} \cdot  \indicator{n \textrm{ is odd}} \, .
\]
\end{theorem}

\begin{proof}
The proof is by induction on~$n$.
The base cases $n = 1$ and $n = 2$ are easy to check.
For~$n \ge 3$,
assume the result is true for~$n - 1$.

If $c$ is the center of the star, 
then $\PP(S_n, c) = \frac{1}{n}$.
If $\ell$ is a leaf, then 
\[
  \PP(S_n, \ell) 
	  = \frac{1}{n} + \frac{n - 1}{n} \Bigl( 1 - \PP(S_{n - 1}) \Bigr)
		= 1 - \frac{n - 1}{n} \PP(S_{n - 1}) .
\]
Taking the weighted average of the center and the $n - 1$ leaves, we have
\[
  \PP(S_n)
	  = \frac{1}{n} \cdot \frac{1}{n} 
		    + \frac{n - 1}{n} \Bigl( 1 - \frac{n - 1}{n} \PP(S_{n - 1}) \Bigr)
		= 1 - \frac{1}{n} + \frac{1}{n^2} - \frac{(n - 1)^2}{n^2} \PP(S_{n - 1}) .		
\]
By the induction hypothesis,
\begin{align*}
  \PP(S_n) 
	  &= 1 - \frac{1}{n} + \frac{1}{n^2} 
		    - \frac{(n - 1)^2}{n^2} 
				\Bigl( \frac{1}{2} +  \frac{1}{2(n - 1)^2} \cdot \indicator{n \textrm{ is even}} \Bigr) \\
		&= \frac{1}{2} + \frac{1}{2n^2} -  \frac{1}{2n^2} \cdot \indicator{n \textrm{ is even}} \\
		&= \frac{1}{2} +  \frac{1}{2n^2} \cdot \indicator{n \textrm{ is odd}} \, ,
\end{align*}
which completes the induction.
\end{proof}

We next exactly compute the probability of winning for a path.

\begin{theorem}  \label{thm:random paths}
If $n$ is an integer such that $n \ge 3$, then 
\[
   \PP(P_n) = \frac{1}{2} + \frac{1}{6n}.
\] 
\end{theorem}
\begin{proof}
For convenience, define $\PP(P_0)$ to be~$\frac{1}{2}$.
Define the function~$\delta$ by 
$\delta(0) = - \frac{1}{6}$, $\delta(1) = \frac{1}{3}$, $\delta(2) = - \frac{1}{6}$, 
and $\delta(k) = 0$ for~$k \ge 3$.
We will prove, for every nonnegative integer~$n$, that $n \PP(P_n) = \frac{1}{2} n + \frac{1}{6} + \delta(n)$.
The proof is by strong induction on~$n$.
The base cases~$n \le 2$ are easy to check.
For $n \ge 3$, assume the result is true for every nonnegative integer less than~$n$.
We will prove the result for~$n$.

We may assume that the vertices of the path in order are $1$, $2$, \dots, $n$.
By the induction hypothesis, for every vertex~$k$, we have
\begin{align*}
  \PP(P_n, k)
		&= 1 - \frac{k - 1}{n} \PP(P_{k - 1}) - \frac{n - k}{n} \PP(P_{n - k})  \\
		&= 1 - \frac{1}{n} \Bigl( \frac{1}{2} (k - 1) + \frac{1}{6} + \delta(k - 1) \Bigr)
		     - \frac{1}{n} \Bigl( \frac{1}{2} (n - k) + \frac{1}{6} + \delta(n - k) \Bigr) \\
		&= \frac{1}{2} + \frac{1}{6n} 
		     - \frac{\delta(k - 1)}{n} - \frac{\delta(n - k)}{n} \, .		
\end{align*}			
Averaging over all vertices gives
\begin{align*}
  \PP(P_n) 
	  &= \frac{1}{n} \sum_{k = 1}^n \PP(P_n, k)  \\
		&= \frac{1}{2} + \frac{1}{6n} 
		     - \frac{1}{n^2} \sum_{k = 1}^n \delta(k - 1) 
				 - \frac{1}{n^2} \sum_{k = 1}^n \delta(n - k)  \\
		&= \frac{1}{2} + \frac{1}{6n} \, .		\qedhere
\end{align*}
\end{proof}

We again conjecture that stars and paths have the extreme win probabilities
(though now stars are the lower bound and paths are the upper bound).

\begin{conjecture}\label{conj:all random}
For every tree~$T$ on $n$~vertices, $\PP(S_n) \leq \PP(T) \leq \PP(P_n)$.
%Among all trees on $n$ vertices, 
%the largest win probability for the first player occurs for the path~$P_n$, and the smallest win probability occurs for the star~$S_n$.
\end{conjecture}

We have confirmed the conjecture for~$n \le 20$.  Conjecture~\ref{conj:all random} implies that the probability of the first player winning is always at least $\frac{1}{2}$ and approaches $\frac{1}{2}$ on every sequence of trees as $n \to \infty$.

%\begin{theorem}
%If $T$ is a tree with $n$ vertices, then 
%\[
%  \frac{1}{3} \leq \PP(T) \leq \frac{2}{3} + \frac{1}{3n} \, .
%\]
%\end{theorem}
%\begin{proof}
%The proof is by induction.
%Assume the result is true for all trees of order less than~$n$, and let $T$ be a tree of order~$n$.
%By the induction hypothesis,
%\begin{align*}
%\PP(T) & =  1 - \frac{1}{n^2} \sum_{v} \sum_{T'} \order{T'} \PP(T')  \\
%& \in \Bigl[ 1 - \frac{1}{n^2} \sum_{v} \sum_{T'} \Bigl(\frac{2}{3} \order{T'} + \frac{1}{3}\Bigr), \, 
%             1 - \frac{1}{n^2} \sum_{v} \sum_{T'} \frac{1}{3} \order{T'} \Bigr] \\
%& = \Bigl[ 1 - \frac{1}{n^2} \Bigl(\frac{2n(n - 1)}{3} + \frac{2(n - 1)}{3} \Bigr), \,
%           1 -  \frac{1}{n^2} \Bigl(\frac{n(n - 1)}{3} \Bigr) \Bigr] \\
%& = \Bigl[ \frac{1}{3} + \frac{2}{3n^2}, \, \frac{2}{3} + \frac{1}{3n}\Bigr] \\
%& \subset \Bigl[ \frac{1}{3}, \, \frac{2}{3} + \frac{1}{3n}\Bigr],
%\end{align*}
%as needed.
%\end{proof}

The next theorem makes progress toward Conjecture~\ref{conj:all random}.  
We show that both players win with probability between $\frac{13}{30}$ and~$\frac{17}{30}$ (for~$n \ge 2$).

\begin{theorem} 
Let $T$ be a tree with $n$ vertices.  
If $n \ge 2$, then
\[
  \frac{13}{30} < \PP(T) < \frac{17}{30} \, .
\]
\end{theorem}

\begin{proof}
Define the function~$\delta$ by 
$\delta(1) = \frac{4}{15}$, $\delta(2) = - \frac{1}{6}$, $\delta(3) = \frac{1}{15}$,
$\delta(4) = - \frac{1}{30}$, $\delta(5) = \frac{2}{15}$, and $\delta(k) = 0$ for~$k \ge 6$.
Define the function~$\Delta$ by 
$\Delta(1) = \frac{1}{3}$, $\Delta(2) = 0$, $\Delta(3) = \frac{2}{45}$, 
and $\Delta(k) = 0$ for~$k \ge 4$.
We will prove, for every positive integer~$n$, the more precise inequality
\[
  \frac{13}{30} + \frac{3}{10n} + \frac{\delta(n)}{n}
	  \le \PP(T)
		\le \frac{17}{30} - \frac{1}{6n} + \frac{4}{15 n^2} + \frac{\Delta(n)}{n} \, .
\]
The proof is by strong induction on~$n$.
The base cases~$n \le 5$ are easy to check.  
For $n \ge 6$, assume the inequality holds for every tree with fewer than $n$~vertices.
We will prove the inequality for tree~$T$.

First we prove the upper bound on~$\PP(T)$.
By the induction hypothesis, 
\begin{align*}
\PP(T) 
  &= 1 - \frac{1}{n^2} \sum_{v} \sum_{T'} \order{T'} \PP(T') \\
  &\le 1 - \frac{1}{n^2} \sum_v \sum_{T'} 
	       \Bigl( \frac{13}{30} \order{T'} + \frac{3}{10} + \delta(\order{T'}) \Bigr) \\
	&= 1 - \frac{13n(n - 1)}{30 n^2} - \frac{6(n - 1)}{10n^2} 
	     - \frac{1}{n^2} \sum_v \sum_{T'} \delta(\order{T'}) \\
	&= \frac{17}{30} - \frac{1}{6n} + \frac{3}{5n^2} 
	                 - \frac{1}{n^2} \sum_v \sum_{T'} \delta(\order{T'}) \\
	&= \frac{17}{30} - \frac{1}{6n} + \frac{3}{5n^2} 
		    - \frac{1}{n^2} \Bigl( 
				    \frac{4}{15} \ell_1(T) - \frac{1}{6} \ell_2(T) 
						  + \frac{1}{15} \ell_3(T) - \frac{1}{30} \ell_4(T) + \frac{2}{15} \ell_5(T)
						            \Bigr) \\
	&= \frac{17}{30} - \frac{1}{6n} + \frac{3}{5n^2} 
	     - \frac{1}{30n^2} 
	    \Bigl( 8 \ell_1(T) - 5 \ell_2(T) + 2 \ell_3(T) - \ell_4(T) + 4 \ell_5(T) \Bigr) .							
\end{align*}

We claim that $8 \ell_1(T) - 5 \ell_2(T) + 2 \ell_3(T) - \ell_4(T) + 4 \ell_5(T) \ge 10$.
For short, write $\ell_k$ for~$\ell_k(T)$.
If $T$ is the unique exceptional tree from Lemma~\ref{lemma:ell}(g) 
(shown in Figure~\ref{fig:counter-example}),
then $\ell_1 = 3$, $\ell_2 = 2$, $\ell_3 = 0$, $\ell_4 = 2$, and $\ell_5 = 3$,
which means
\[
  8 \ell_1 - 5 \ell_2 + 2 \ell_3 - \ell_4 + 4 \ell_5
	  = 24 - 10 + 0 - 2 + 12
		= 24.
\]
Hence we may assume that $T$ is not the exceptional tree.
By Lemma~\ref{lemma:ell}(g),
we have $\ell_1 + \ell_3 \ge \ell_2 + \ell_4$, 
which means
\[
  8 \ell_1 - 5 \ell_2 + 2 \ell_3 - \ell_4 + 4 \ell_5
	  \ge 7 \ell_1 - 4 \ell_2 + \ell_3 + 4 \ell_5 .
\]
If $\ell_1 = 2$, then $T$ is a path (Lemma~\ref{lemma:2-3 leaves}(a)), 
so $\ell_2 = \ell_3 = \ell_4 = \ell_5 = 2$ (Lemma~\ref{lemma:2-3 leaves}(b)), which means
\[
  7 \ell_1 - 4 \ell_2 + \ell_3 + 4 \ell_5 
	  = 14 - 8 + 2 + 8
		= 16.
\] 
If $\ell_1 \ge 4$, then because $\ell_1 \ge \ell_2$ (Lemma~\ref{lemma:ell}(f)), we have
\[
  7 \ell_1 - 4 \ell_2 + \ell_3 + 4 \ell_5 
	  \ge 7 \ell_1 - 4 \ell_2
		\ge 3 \ell_1
		\ge 12.
\]
Hence we may assume that $\ell_1 = 3$.
If $\ell_2 \le 2$, then
\[
  7 \ell_1 - 4 \ell_2 + \ell_3 + 4 \ell_5 
	  \ge 7 \ell_1 - 4 \ell_2
		= 21 - 4 \ell_2
		\ge 21 - 8
		= 13.
\]
Hence we may assume that $\ell_2 = 3$.
If $\ell_3 \ge 1$, then
\[
  7 \ell_1 - 4 \ell_2 + \ell_3 + 4 \ell_5 
    \ge 7 \ell_1 - 4 \ell_2 + \ell_3
		= 21 - 12 + \ell_3
		\ge 21 - 12 + 1
		= 10.
\]
Hence we may assume that $\ell_3 = 0$.
By Lemma~\ref{lemma:2-3 leaves}(e),
the only tree with $\ell_1 = 3$, $\ell_2 = 3$, and $\ell_3 = 0$ 
is the spider~$\spider_{2, 2, 2}$, shown in Figure~\ref{fig:3-3-0}.
For that tree, $\ell_5 = 3$, which means
\[
  7 \ell_1 - 4 \ell_2 + \ell_3 + 4 \ell_5 
	  = 21 - 12 + 0 + 12
		= 21.  
\]
In every case, 
we have proved the claim 
$8 \ell_1(T) - 5 \ell_2(T) + 2 \ell_3(T) - \ell_4(T) + 4 \ell_5(T) \ge 10$.

Plugging the claim into our previous inequality for $\PP(T)$ gives
\begin{align*}
\PP(T) 
  &\le \frac{17}{30} - \frac{1}{6n} + \frac{3}{5n^2} 
	     - \frac{1}{30n^2} 
	    \Bigl( 8 \ell_1(T) - 5 \ell_2(T) + 2 \ell_3(T) - \ell_4(T) + 4 \ell_5(T) \Bigr) \\
	&\le \frac{17}{30} - \frac{1}{6n} + \frac{3}{5n^2} - \frac{10}{30 n^2} \\
	&= \frac{17}{30} - \frac{1}{6n} + \frac{4}{15 n^2} \, , 
\end{align*}
which is the desired upper bound.

Next we prove the lower bound on~$\PP(T)$.
By the induction hypothesis,
\begin{align*}
\PP(T) 
  &= 1 - \frac{1}{n^2} \sum_{v \in V(T)} \sum_{T'} \order{T'} \PP(T') \\
  &\ge 1 - \frac{1}{n^2} \sum_v \sum_{T'} \Bigl( 
	       \frac{17}{30} \order{T'} - \frac{1}{6} 
				   + \frac{4}{15\order{T'}} + \Delta(\order{T'}) \Bigr) \\
	&= 1 - \frac{17n(n - 1)}{30 n^2} + \frac{2(n - 1)}{6n^2} 
	     - \frac{4}{15n^2} \sum_v \sum_{T'} \frac{1}{\order{T'}} 
			 - \frac{1}{n^2} \sum_v \sum_{T'} \Delta(\order{T'}) \\
	&= \frac{13}{30} + \frac{17}{30n} 
	      - \frac{4}{15n^2} \sum_v \sum_{T'} \frac{1}{\order{T'}}
      	+ \frac{n - 1}{3n^2} 
	      - \frac{1}{n^2} \sum_v \sum_{T'} \Delta(\order{T'}) \\
	&= \frac{13}{30} + \frac{17}{30n} 
       - \frac{4}{15n^2} \sum_v \sum_{T'} \frac{1}{\order{T'}}
			 + \frac{n - 1}{3n^2}
		   - \frac{1}{n^2} \Bigl( \frac{1}{3} \ell_1(T) + \frac{2}{45} \ell_3(T) \Bigr) 
			     \\
	&= \frac{13}{30} + \frac{17}{30n} 
	     - \frac{4}{15n^2} \sum_v \sum_{T'} \frac{1}{\order{T'}}
			 + \frac{1}{3 n^2} \Bigl( n - 1 - \ell_1(T) - \frac{2}{15} \ell_3(T) \Bigr) \, .
\end{align*}
By Lemma~\ref{lemma:ell}(c) and Lemma~\ref{lemma:ell}(d), 
the expression in parentheses is nonnegative:
\[
  \ell_1(T) + \frac{2}{15} \ell_3(T) 
	  = \frac{1}{2} \ell_1(T) + \frac{1}{2} \ell_{n - 1}(T) + \frac{2}{15} \ell_3(T)
		\le \frac{1}{2} \sum_{k = 1}^{n - 1} \ell_k(T)
		= \frac{1}{2} \cdot 2(n - 1)
		= n - 1 .
\]
Hence
\[
  \PP(T) 
    \ge \frac{13}{30} + \frac{17}{30n} 
		      - \frac{4}{15 n^2} \sum_v \sum_{T'} \frac{1}{\order{T'}} \, . 		
\]

We claim that the double sum can be bounded as follows:
\[
  \sum_v \sum_{T'} \frac{1}{\order{T'}} \le n .
\]
Given a vertex~$v$ and one of its neighboring vertices~$w$,
recall from the proof of Lemma~\ref{lemma:ell}(c) that the tree~$T_{v, w}$ is 
the component of $T \smallsetminus v$ that contains~$w$.
We can rewrite the inequality as
\[
   \sum_{\substack{v, w\\ v \sim w}} \frac{1}{\order{T_{v, w}}} \le n .
\]
For every edge~$\{ v, w \}$ of~$T$, we have
\[
  \frac{1}{\order{T_{v, w}}} + \frac{1}{\order{T_{w, v}}}
	  = \frac{\order{T_{v, w}} + \order{T_{w, v}}}{\order{T_{v, w}} \cdot \order{T_{w, v}}} 
		= \frac{n}{\order{T_{v, w}} \cdot \order{T_{w, v}}} 
		\le \frac{n}{n - 1} \, .
\]
Hence we can sum over every edge:
\[
  \sum_{\substack{v, w\\ v \sim w}} \frac{1}{\order{T_{v, w}}} 
	  \le (n - 1) \cdot \frac{n}{n - 1} = n , 
\]
which establishes the claimed bound on the double sum. 

Plugging this bound on the double sum into our previous inequality for $\PP(T)$ gives
\[
  \PP(T) 
	  \ge \frac{13}{30} + \frac{17}{30n} 
		      - \frac{4}{15n^2} \sum_v \sum_{T'} \frac{1}{\order{T'}}
		\ge \frac{13}{30} + \frac{17}{30n} - \frac{4}{15n^2} \cdot n 
		= \frac{13}{30} + \frac{3}{10n} \, ,
\]
which is the desired lower bound.
\end{proof}

\section{Other questions}
\label{sec:concluding remarks}

The main question arising from our work is to resolve
Conjectures~\ref{conj:half random tree bounds} and~\ref{conj:all random}, namely, 
to show that the path and star have the extreme probabilities of winning 
among all trees of fixed order in both the semi-random and all-random models.
In this section we outline a number of other possible directions of future research.

\subsection{Other models}

The rules of the game studied in this paper can be varied or extended in numerous ways.  We mention a few explicit cases that are similar in spirit and might lead in interesting directions.

\subsubsection*{Guessing edges}

Rather than choosing a vertex, players might alternate choosing \emph{edges}, with the goal being to isolate the randomly selected target vertex.  This version of the game is reminiscent of the random process called \emph{cutting down trees}, which was introduced by Meir and Moon \cite{MM} and is an active area of research (see, e.g., \cite{ABetal}).

\subsubsection*{Searching on general graphs}

In \cite{EZ-K-S}, the binary search algorithm on trees was extended to a model on general connected graphs: upon querying a vertex $v$, the player either is informed that $v$ is the target or is given an edge out of $v$ that lies on a shortest path from $v$ to the target.  One could extend the two-player version of the game to this model.  On some graphs, and unlike for trees, the first player may be at a marked disadvantage in this generalization: for example, on the complete graph $K_n$, the first player wins with probability $\frac{1}{n}$.  

\subsubsection*{Mis\`ere search}

In the mis\`ere version of the game, the target vertex is poisoned, and the player who finds it loses.  For example, the one-vertex tree is a first player loss, the two-vertex tree is a fair game, and the first player wins on the three-vertex tree with probability $\frac{2}{3}$ if they choose the central vertex and with probability $\frac{1}{3}$ if they choose a leaf.  

\subsubsection*{Continuous models}

The following continuous model of binary search also admits a two-player version: a target point is selected uniformly at random in the real interval $[a, b]$, and a player wins if they pick a point within some fixed distance $\varepsilon > 0$ of the target; non-winning guesses are informed whether the guess is too large or too small.  Higher-dimensional analogues are also possible: a guess might consist of a hyperplane, with the goal being to bound the target point in a region of sufficiently small volume; or a guess might consist of a point and a hyperplane through that point, with the goal of choosing a point within $\varepsilon$ of the target.

\subsection{Probability distributions}

In the models where one or both players play randomly, 
the set of probabilities of a first-player win lies in some nontrivial interval~$[a, b]$.  
As the order~$n$ of the trees goes to infinity, 
is the set of these probabilities dense in some interval?  
Is there a limiting distribution of probabilities of winning, 
where (for example) we think of drawing a vertex-labeled tree uniformly at random?

\subsection{Random vs.\ random as a random process}

The all-random game (Section~\ref{sec:random}) involves no strategy, 
so we can view it as a purely random process.  
Namely, we are given a tree~$T$,
one vertex of which has been randomly selected as the target.  
At each step, a vertex~$v$ is chosen at random;
if $v$ is not the target, 
then the tree is cut down to the subtree of~$T \smallsetminus v$ in which the target lies.
The \emph{stopping time} is the number of steps it takes to hit the target.
Our problem can be restated as bounding the probability that the stopping time is odd.  

For the star~$S_n$, the stopping time has expected value about~$n/3$.
For the path~$P_n$, the stopping time has expected value about~$2 \ln n$.
Does every tree with $n$~vertices have expected stopping time between 
those of the path and the star?
For a given tree,
is the probability distribution of the stopping time unimodal?
Answering these questions might help resolve Conjecture~\ref{conj:all random}.

\subsection{Fixed target}

The all-random game (Section~\ref{sec:random}) has three sources of randomness: 
the target vertex, the first player, and the second player.  
What if instead the target vertex were fixed (so that effectively we have a rooted tree)? 
Does the first player always win with probability at least~$\frac{1}{2}$?
After looking at small trees (up to five vertices), 
we tentatively conjecture that the answer is yes.
If the target is a leaf, then a simple induction shows that the game is fair
(win probability exactly~$\frac{1}{2}$). 

\subsection{Relation with the gold grabber game}

In \cite{SS}, Seacrest and Seacrest consider the following (non-random, complete information) game on trees: some coins of various values are distributed on the vertices of a tree, and players take turns selecting a leaf and collecting the coins on the chosen leaf.  Their main theorem is that on a tree with an even number of vertices, the first player can always guarantee to acquire at least half the total value of coins.
Our game could be rephrased in similar language, with the target having a coin of positive value and all other vertices having coins of value $0$.  In this rephrasing, the key difference between our game and the game considered by Seacrest and Seacrest is that the players do not know the distribution of the coins.  Are there interesting variations of the binary search game involving more elaborate distributions of weight?

\subsection{Other constraints on limb numbers}

Lemma~\ref{lemma:ell} gives a number of constraints on the limb numbers~$\ell_i(T)$, which are used for the bounds in several results in Sections~\ref{sec:mixed} and~\ref{sec:random}.  These relations are not exhaustive; for example, the proof of Lemma~\ref{lemma:ell}(g) may be extended to show that
\[
\ell_1(T) + \ell_3(T) + \dots + \ell_{2k - 1}(T) 
  \geq \ell_2(T) + \ell_4(T) + \dots + \ell_{2k}(T)
\]
for every tree~$T$ of order at least~$4k$.  
On the other hand, for any positive integers~$a$ and~$b$, 
one can construct a tree~$T$ such that $\ell_2(T) = a$ and $\ell_3(T) = b$: 
if $a = b = 1$ then take $T$ to be the spider~$\spider_{2, 1, 1}$, 
shown at the top of Figure~\ref{fig:ell2 ell3},
while if $a + b > 2$ then take $T$ to be the bottom tree in Figure~\ref{fig:ell2 ell3}
with $a$~branches on the left and $b$~branches on the right.
\begin{figure}
\begin{center}
\scalebox{.75}{
\begin{tikzpicture}[node distance = 2cm, v/.style={circle,draw}]
\node[v] (1) {};
\node[v] (2) [right of = 1] {};
\node[v] (3) [right of = 2] {};
\node[v] (4) [above right of = 3] {};
\node[v] (5) [below right of = 3] {};
\draw (1) -- (2) -- (3) -- (4);
\draw (3) -- (5);
\end{tikzpicture}}
\end{center}
\begin{center}
\scalebox{.75}{\begin{tikzpicture}[v/.style={circle,draw}]
\node (c) at (0, 0) [v] {};
\node (11) at ({-sqrt(3)}, -1) [v] {};
\node (12) at ({-2*sqrt(3)}, -2) [v] {};
\node (21) at ({-sqrt(3)}, 1) [v] {};
\node (22) at ({-2*sqrt(3)}, 2) [v] {};
\node (31) at ({- 2 * cos(15)}, { 2 * sin(15)}) [v] {};
\node (32) at ({- 4 * cos(15)}, { 4 * sin(15)}) [v] {};
\draw  (c)  --   (31) --  (32) ;
\draw  (c)  --   (21) --  (22) ;
\draw  (c)  --   (11) --  (12) ;
\node at (-3, -.4) {$\vdots$};
\node (41) at ({2 * cos(30)}, {2 * sin(30)}) [v] {};
\node (42) at ({2 * cos(30) + 2 * cos(37)}, {2 * sin(30) + 2*sin(37)}) [v] {};
\node (43) at ({2 * cos(30) + 2 * cos(23)}, {2 * sin(30) + 2*sin(23)}) [v] {};
\node (51) at ({ 2 * cos(15)}, { 2 * sin(15)}) [v] {};
\node (52) at ({ 2 * cos(15) + 2 * cos(22)}, { 2 * sin(15) + 2 * sin(22)}) [v] {};
\node (53) at ({ 2 * cos(15) + 2 * cos(8)}, { 2 * sin(15) + 2 * sin(8)}) [v] {};
\node (61) at ({2 * cos(30)}, {-2 * sin(30)}) [v] {};
\node (62) at ({2 * cos(30) + 2 * cos(37)}, {-2 * sin(30) - 2*sin(37)}) [v] {};
\node (63) at ({2 * cos(30) + 2 * cos(23)}, {-2 * sin(30) - 2*sin(23)}) [v] {};
\draw  (c)  --   (41) --  (42) ;
\draw (41) -- (43);
\draw  (c)  --   (51) --  (52) ;
\draw (51) -- (53);
\draw  (c)  --   (61) --  (62) ;
\draw (61) -- (63);
\node at (3, -.4) {$\vdots$};
\end{tikzpicture}}
\end{center}
\caption{Trees with specified values for $\ell_2$ and $\ell_3$}
\label{fig:ell2 ell3}
\end{figure}
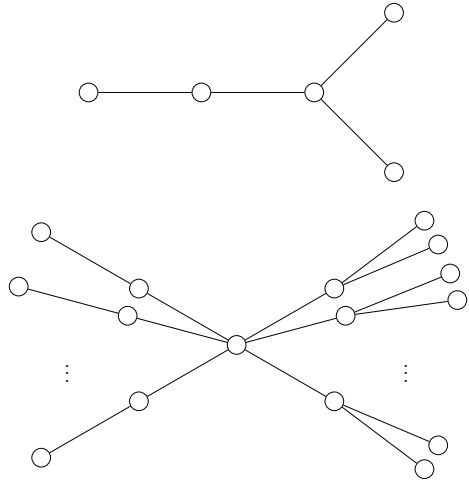
What else can one say about the constraints satisfied by the tuples~$(\ell_1, \dots, \ell_k)$?  For example, are the tuples~$(\ell_1, \dots, \ell_k)$ the set of lattice points in a polytope (at least for $n$ sufficiently large)?

\section*{Acknowledgments}

We first encountered a version of this problem on 
the Art of Problem Solving forum~\cite{AoPS}.
We thank Stephen Merriman for proposing the initial version of this problem.  
We thank Zachary Abel and Alison Miller for performing some helpful computer calculations.
We are especially grateful to Simon Rubinstein-Salzedo 
for asking about playing randomly, 
for formulating the tree version of the problem, 
and for writing a useful computer program.  
We thank the two referees for their numerous helpful comments.  
We used the SageMath computer algebra system~\cite{Sage} to test our conjectures.
The second author's research was supported in part by a ORAU Powe Award and 
a Simons Collaboration Grant (634530).

\end{document}